\documentclass{amsart}
\usepackage{graphicx,fancybox,amssymb,color,amsmath,amssymb,amsthm,amsfonts}
\usepackage{latexsym}

\newtheorem{theorem}{Theorem}[section]

\newtheorem{lemma}[theorem]{Lemma}
\newtheorem{proposition}[theorem]{Proposition}

\theoremstyle{definition}
\newtheorem{definition}[theorem]{Definition}
\theoremstyle{remark}
\newtheorem{remark}[theorem]{Remark}
\theoremstyle{remark}

\numberwithin{equation}{section}
\newcommand{\E}{{\rm E}}
\newcommand{\Var}{{\rm Var}}

\usepackage{color,soul}
\usepackage{dsfont}
\newcommand{\RR}{\mathds{R}}

\newcommand{\qq}{\mbox{\fontfamily{phv}\selectfont Q}}

\newcommand{\pp}{\mbox{\fontfamily{phv}\selectfont P}}
\newcommand{\rr}{\mbox{\fontfamily{phv}\selectfont R}}

\newcommand{\hh}{\mbox{\fontfamily{phv}\selectfont H}}

\newcommand{\uuu}{\mathbb{Y}}
\newcommand{\vvv}{\mathbb{X}}

\newcommand{\gen}{\mbox{\fontfamily{phv}\selectfont A}}

\newcommand{\aaa}{\mathbb{A}}

\newcommand{\ddd}{\mathbb{D}}
\newcommand{\jjj}{\mathbb{J}}

\newcommand{\hhh}{\mathbb{H}}
\newcommand{\ppp}{\mathbb{P}}
\newcommand{\qqq}{\mathbb{Q}}

\newcommand{\mmm}{\mathbb{M}}
\newcommand{\rrr}{\mathbb{R}}
\newcommand{\eee}{\mathbb{E}}
\newcommand{\fff}{\mathbb{F}}
\newcommand{\ggen}{\mathbb{A}}
\newcommand{\ttt}{\mathbb{T}}
\newcommand{\AQH}{\mathcal{Q}}

\newcommand{\calP}{\mathcal{P}}

\newcommand{\calF}{\mathcal{F}}

\newcommand{\lp}{\circ}

\renewcommand{\lp}{ }

\date{Printed \today. (\jobname.tex)}

\title[Infinitesimal generators for  polynomial processes]{Infinitesimal generators for a class of polynomial processes}

\author{
W{\l}odzimierz  Bryc
}

\address{
Department of Mathematics,
University of Cincinnati,
PO Box 210025,
Cincinnati, OH 45221--0025, USA}
\email{Wlodzimierz.Bryc@UC.edu}

\author{Jacek Weso{\l}owski}
\address{ Faculty of Mathematics and Information Science
Warsaw University of Technology pl. Politechniki 1 00-661
Warszawa, Poland}
\email{wesolo@alpha.mini.pw.edu.pl}


\usepackage{float}
\usepackage{dsfont,amsfonts,amsmath,pgf}
\usepackage{fancybox}

 \usepackage{framed}
\colorlet{shadecolor}{gray!30}
\newenvironment{ana}
  {\begin{leftbar}
  \begin{shaded} }
{  \end{shaded}\end{leftbar}}

 \newcommand{\arxiv}[1]{\begin{ana} #1\end{ana}}

\begin{document}
\maketitle

\begin{abstract}
We study  the infinitesimal generators  of evolutions of linear mappings on the space of polynomials, which  correspond to a special class of Markov processes with polynomial regressions called quadratic harnesses.  We relate the infinitesimal generator to the unique solution of a certain commutation equation, and we use the commutation equation to find an explicit formula
for the infinitesimal generator  of free quadratic harnesses.
\end{abstract}

\arxiv{This is an expanded (arxiv) version of the paper.}

\section{Introduction}
In this paper we study properties of evolutions of degree-preserving linear mapping on the space of polynomials. We analyze these mappings in a self-contained algebraic language assuming a number of algebraic properties  which were abstracted out from some special properties of a family of Markov processes called quadratic harnesses. We therefore begin with a review of the relevant theory of Markov processes that motivates our assumptions.

Markov processes that motivate our theory have polynomial regressions with respect to the past $\sigma$-fields. Such Markov processes enjoy many interesting properties and appeared in numerous references
\cite{bakry2003characterization,barrieu2006iterates,Bryc-Wesolowski-03,Bryc-Wesolowski-08,cuchiero2012polynomial,Di-Nardo-2013,Di-Nardo-Oliva-2013,schoutens2000stochastic,schoutens1998levy,sengupta2000time,sengupta2008markov,sengupta2001finitely,sole2008orthogonal,sole2008time,szablowski2012evy,szablowski2013stationary,szablowski2014markov}.
As observed by Cuchiero in \cite{cuchiero2011affine} and Szab\l owski \cite{szablowski2012markov},  transition probabilities $P_{s,t}(x,dy)$ of  such a process on an infinite state space define  the  family of linear
transformations $(\pp_{s,t})_{0\leq s\leq t}$ that map  the linear space $\calP=\calP (\RR)$ of all polynomials  in variable
$x$ into itself. The crucial property that holds in many interesting examples is that  transformations $\pp_{s,t}$ do not increase the degree of a polynomial, that is $\pp_{s,t}:\calP_{\le k}\to\calP_{\le k}$, where $\calP_{\le k}$ denotes a linear space of polynomials of degree $\le k$, $k=0,1,\ldots$. This will be the basis for our algebraic approach.
Cuchiero in \cite{cuchiero2011affine}, see also
\cite{cuchiero2012polynomial} introduced the term "polynomial process"  to  denote such a process in the time-homogeneous case; we will use this term more broadly to denote the family of operators rather than a Markov process.
That is, we  adopt the point of view  that the linear mappings $\pp_{s,t}$ of $\calP$ can be analyzed ``in abstract" without explicit
reference to the underlying Markov process and the transition operators.

We note that operators $\pp_{s,t}$ are well-defined whenever the support of $X_s$ is infinite. So, strictly speaking, operators $\pp_{0,t}$ are well defined only for the so called Markov families that can be started at an infinite number of values of $X_0$. However, it will  turn out that in the cases that we are interested in, even if a Markov process starts with $X_0=0$ we can pass to the limit in $\pp_{s,t}$ as $s\to0$ and in this way define a unique degree-preserving mapping $\pp_{0,t}:\calP\to\calP$.

With the above in mind, we introduce the following definition.
\begin{definition}\label{Def-PPP} Suppose $\calP$ is the linear space of polynomials in one variable $x$. A polynomial process is a family of linear maps $\{\pp_{s,t}:\calP\to\calP,\;0\le s\le t\}$  with the following properties:
\begin{enumerate}
  \item for  $k=0,1,\dots$ and $0\leq s\leq t$, $$\pp_{s,t}(\calP_{\leq k})=\calP_{\leq k},$$
  \item $\pp_{s,t}(1)=1$,
  \item for $0\leq s\leq t\leq u$
   \begin{equation}
     \label{evolution_eqtn-1}
     \pp_{s,t}\circ\pp_{t,u}=\pp_{s,u}\,.
   \end{equation}
\end{enumerate}
\end{definition}

Our next task is to abstract out properties of polynomial processes   that correspond to
  a special class of such Markov processes with linear regressions and quadratic conditional variances under the two-sided conditioning.
 The two-sided linearity of regression, which is sometimes called the ``harness property'',
 takes the following  form:
 \begin{equation}
\label{EQ:LR} \E({X_t}|X_s, X_u)=\frac{u-t}{u-s}
X_s+\frac{t-s}{u-s} X_u,  \mbox{ $0\leq s<t<u$,}
\end{equation}
see e.~g., \cite[formula (2)]{Mansuy-Yor-05}.

We will also assume that the conditional second moment of $X_t$ given $X_s,\,X_u$ is a second degree polynomial in the variables $X_s$ and $X_u$.  The conclusion of \cite[Theorem 2.2]{Bryc-Matysiak-Wesolowski-04} says  that there are five numerical constants
 $\eta,\theta\in\RR$, $\sigma,\tau\geq 0$, and $\gamma\leq 1+2\sqrt{\sigma\tau}$ such that
for all $0\le s<t<u$,
\begin{multline}\label{EQ:q-Var}
\Var(X_t|X_s, X_u) %\\
= \frac{(u-t)(t-s)}{u(1+\sigma s)+\tau-\gamma s}\left( 1+\eta \frac{u X_s-s X_u}{u-s} +\theta\frac{X_u-X_s}{u-s}\right. \\
 \left.
+ \sigma
\frac{(u X_s-s X_u)^2}{(u-s)^2}+\tau\frac{(X_u-X_s)^2}{(u-s)^2}
-(1-\gamma)\frac{(X_u-X_s)(u X_s-s X_u)}{(u-s)^2} \right).
\end{multline}

Such processes, when standardized, are called quadratic harnesses. Typically they are uniquely determined by five constants $\eta,\,\theta,\, \sigma,\,\tau,\,\gamma$ from \eqref{EQ:q-Var}.
We will also assume that $(X_t)$ is a martingale in a natural filtration. This is a consequence of \eqref{EQ:LR} when $\E(X_t)$ does not depend on $t$, compare \cite[page 417]{Bryc-Wesolowski-03}.

The martingale property is easily expressed in the language of polynomial processes, as it just says that $\pp_{s,t}(x)=x$.
Somewhat more generally, if the polynomials $m_k(x;t)$ in variable $x$, $k\ge 0$,    are martingale polynomials for a Markov process $(X_t)$, that is $\E(m_k(X_t;t)|X_s)=m_k(X_s,s)$, $k\ge 0$, then $\pp_{s,t}(m_k(\cdot;t))(x)=m_k(x;s)$ for $s<t$.

It is harder to deduce the properties of $\pp_{s,t}$ that correspond to   \eqref{EQ:LR} and \eqref{EQ:q-Var}. We will find it useful to  describe the linear mappings of polynomials in the algebraic language, and we will rely on martingale polynomials to complete this task.

\subsection{Algebra $\AQH$ of sequences of polynomials}  We consider the linear space  $\AQH$ of all infinite sequences of polynomials
in variable $x$ with a nonstandard  multiplication    defined as follows. For $\ppp=(p_0,p_1,\dots,p_k,\dots)$  and
 $\qqq=(q_0,q_1,\dots,q_k,\dots)$ in $\AQH$, we define their product
 $\rrr=\ppp \qqq$ as the sequence of polynomials $\rrr=(r_0,r_1,\dots)\in\AQH$ given by
 \begin{equation}\label{def-mult}
  r_k(x)=\sum_{j=0}^{\deg(q_k)} [q_k]_jp_j(x)\,, k=0,1,\dots.
\end{equation}
It is easy to check that algebra $\AQH$ has identity
$$\eee=(1,x,x^2,\dots).$$ In the sequel we will frequently use two special
elements of $\AQH$
    \begin{equation}
  \label{def-fff}
  \fff=(x,x^2,x^3,\dots),
\end{equation}
  and
  \begin{equation}\label{def-ddd}
  \ddd=(0,1,x,x^2,\dots)
  \end{equation}
which is the  left-inverse of $\fff$.

\arxiv{It may be useful to note that $\ppp\fff$ and $\ppp\ddd$ act as shifts:  if $\ppp=(p_0(x),p_1(x),\dots,p_n(x),\dots)$ then
$\ppp\fff=(p_1(x),p_2(x),\dots,p_{n+1}(x),\dots)$ and
 $\ppp\ddd=(0,p_0(x),p_1(x),\dots,p_{n-1}(x),\dots)$.

On the other hand, $\fff\ppp$ is just a multiplication by $x$, i.e., $\fff\ppp=(xp_0(x),xp_1(x),\dots,xp_n(x),\dots)$.
}

Algebra $\AQH$ is isomorphic to the algebra of all  linear mappings
$\calP\to\calP$ under composition: to each $\pp:\calP\to\calP$ we associate a unique sequence of polynomials $\ppp=(p_0,p_1,\dots,p_k,\dots)$ in variable $x$ by setting
$p_k=\pp(x^k)$.  Of course, if $\pp$ does not increase the degree, then $p_k$ is of at most degree $k$.

Under this isomorphism,  the composition $\rr=\pp\circ\qq$ of linear operators $\pp$ and $\qq$ on $\calP$ induces the
multiplication operation  $\rrr=\ppp \lp \qqq$ for  the corresponding sequences of polynomials $\ppp=(p_0,p_1,\dots,p_k,\dots)$
and  $\qqq=(q_0,q_1,\dots,q_k,\dots)$ which was introduced in \eqref{def-mult}.

It is clear that degree-preserving linear mappings of $\calP$   are invertible.
\begin{proposition}\label{P-inv}
If for every $n$ polynomial $p_n$ is of degree $n$ then  $\ppp=(p_0,p_1,\dots)$ has %
multiplicative inverse
$\qqq=(q_0,q_1,\dots)$ and each  polynomial $q_n$ is of degree $n$.
\end{proposition}
\arxiv{
\begin{proof}
Write $p_n(x)=\sum_{k=0}^n a_{n,k}x^k$. The inverse $\qqq=(q_0,q_1,\dots)$ is given by the family of polynomials $q_n$ of
degree $n$ which solve the following recursion:
$$q_0=\frac{1}{a_{0,0}},  \mbox{ and }  q_n=\frac{1}{a_{n,n}} \left(x^n-\sum_{j=0}^{n-1}
a_{n,j}q_j(x)\right) \mbox{ for $n\geq 1$}.$$
It is then clear that $\qqq\lp\ppp=\eee$.

Since both $\ppp$ and $\qqq$ consist of polynomials of degree $n$ at the $n$-th position of the sequence, the corresponding
linear mappings $\pp,\qq$ on the set of polynomials preserve the degree of a polynomial. Since $\qq\circ\pp$ is the identity on
each finite-dimensional space  $\calP_{\leq n}$,  we see that
 $\pp\circ\qq$ is also an identity, so   $\ppp\lp\qqq=\eee$.

\end{proof}
}

In this paper we  study $\pp_{s,t}$ through the corresponding elements $\ppp_{s,t}$ of the algebra $\AQH$. We therefore rewrite
Definition \ref{Def-PPP}  in the language of algebra $\AQH$.

\begin{definition}[Equivalent form of Definition \ref{Def-PPP}]
  \label{D1}
A polynomial process is a family $\{\ppp_{s,t}\in\AQH: 0\leq s\leq t\}$  %
  with the following properties:
\begin{enumerate}
    \item   for $0\leq s\leq t $ and $n=0,1,\dots$, the $n$-th component of $\ppp_{s,t}$ is a polynomial of degree $n$,
    \item $\ppp_{s,t}(\eee-\fff\lp \ddd)=\eee-\fff\lp \ddd$,

  \item for $0\leq  s\leq t \leq u$ we assume
    \begin{equation}
    \label{evolution_eqtn}
    \ppp_{s,t}\lp \ppp_{t,u}=\ppp_{s,u}.
  \end{equation}
\end{enumerate}
\end{definition}

Since special elements \eqref{def-fff} and \eqref{def-ddd} satisfy $(1,0,0,\dots)=\eee-\fff\lp \ddd$,   property (ii)   is  $\ppp_{s,t}(1,0,0,\dots)=(1,0,0,\dots)$, see Definition \ref{Def-PPP}(ii).

\subsection{Martingale polynomials} \label{Sect:MP}
In this section we re-derive \cite[Theorem
1]{szablowski2012markov}.
 We note that by Proposition \ref{P-inv}   each $\ppp_{s,t}$ is an invertible element of $\AQH$. So  from \eqref{evolution_eqtn}
we see that $\ppp_{t,t}=\eee$ for all $t \geq  0$.   In particular, $\ppp_{0,t}$ is invertible and $\mmm_t=\ppp_{0,t}^{-1}$
consists of polynomials $m_k(x;t)$ in variable $x$ of degree $k$. From  \eqref{evolution_eqtn}, we get
$\ppp_{0,s}\lp\ppp_{s,t}\lp\mmm_t=\ppp_{0,t}\lp\mmm_t=\eee$. Multiplying this on the left by $\mmm_s=\ppp_{0,s}^{-1}$, we see
that
\begin{equation}\label{eq-mart}
  \mmm_s=\ppp_{s,t}\lp\mmm_t,
\end{equation}
i.e.,  $\mmm_t$ is a sequence of martingale polynomials for $\{\ppp_{s,t}\}$ and in addition $\mmm_0=\eee$. Conversely,
\begin{equation}
  \label{m2p}
  \ppp_{s,t}=\mmm_s\lp\mmm_t^{-1}.
\end{equation}

Ref. \cite{szablowski2012markov} points out that in general martingale polynomials are not unique. However, any  sequence $\widetilde \mmm_t=(m_0(x;t),m_1(x;t),\dots)$ of martingale polynomials (with  $m_k(x;t)$
of degree $k$ for  $k=0,1,\dots$), still determines uniquely $\ppp_{s,t}$ via $\ppp_{s,t}=\widetilde\mmm_s\lp\widetilde\mmm_t^{-1}$.

\subsection{Quadratic harnesses}\label{Sect:QH}

Recall that our goal is to abstract out the properties of  the (algebraic) polynomial process $\{\ppp_{s,t}\}$ that  correspond to relations \eqref{EQ:LR} and \eqref{EQ:q-Var}. Since this correspondence is not direct, we first give   the
self-contained algebraic definition, and then point out the motivation for such definition.

\begin{definition}\label{D2}
   We will say that a polynomial process  $\{\ppp_{s,t}:0\leq s\leq t\}$ is a quadratic harness with parameters  $\eta,\theta,\sigma,\tau,\gamma$ if the following three conditions hold:
   \begin{enumerate}
     \item (martingale property) $\ppp_{s,t}(\fff\ddd-\fff^2\ddd^2)=\fff\ddd-\fff^2\ddd^2$,
     %$\ppp_{s,t}(0,x,0,0,\dots)=(0,x,0,0,\dots)$.
     \item (harness property) there exists $\vvv\in \AQH$ such that for
         all $t\geq 0$  we have
     \begin{equation}
  \label{harness}
 \ppp_{0,t} \fff\ = (\fff+t\vvv)\lp\ppp_{0,t},
\end{equation}
\item (quadratic harness property) element $\vvv\in \AQH$  introduced in \eqref{harness} satisfies the following quadratic equation
\begin{equation}
  \label{q-harness}
\vvv\lp\fff-\gamma \fff\lp\vvv =\eee+\eta\fff+\theta\vvv+\sigma \fff^2+\tau \vvv^2.
\end{equation}

  \end{enumerate}
\end{definition}

\begin{remark} Using  martingale polynomials $\mmm_t=\ppp_{0,t}^{-1}$ from Section \ref{Sect:MP}, we can re-write
 assumption \eqref{harness} as
 \begin{equation}
  \label{harness-M}
  \fff\lp\mmm_t=\mmm_t\lp(\fff+t\vvv).
\end{equation}
\end{remark}

We now give a brief explanation how the algebraic properties in conditions (i)-(iii) of Definition \ref{D2} are related to the corresponding properties of a (polynomial) Markov process.

\subsubsection{Motivation for the martingale property}
For a Markov process $(X_t)$ with   the natural filtration
      $(\calF_t)$ this takes a more familiar form $E(X_t|\calF_s)=X_s$. Translated back into the language of linear operators on polynomials,
   this is $\pp_{s,t}(x)=x$. In the language of algebra $\AQH$ this is  $\ppp_{s,t}(0,x,0,0,\dots)=(0,x,0,0,\dots)$.   To get the final form of condition (i) we note that $\fff\ddd-\fff^2\ddd^2=(0,x,0,0,\dots)$.
\subsubsection{Motivation for harness property \eqref{harness}}
 For a Markov processes $(X_t)$ with martingale polynomials $m_n(x;t)$,  property \eqref{EQ:LR} implies that
$$E(X_t m_n(X_t;t)|X_s) = E(X_t m_n(X_u;u)|X_s) =  E(E(X_t|X_s,X_u) m_n(X_u;u)|X_s)$$
$$
=\frac{u-t}{u-s} X_s E(m_n(X_u;u)|X_s)+\frac{t-s}{u-s} E(X_um_n(X_u;u)|X_s)$$
$$
=\frac{u-t}{u-s} X_s m_n(X_s;s)+\frac{t-s}{u-s} E(X_um_n(X_u;u)|X_s)$$

The resulting identity
$$E(X_t m_n(X_t;t)|X_s) =\frac{u-t}{u-s} X_s m_n(X_s;s)+\frac{t-s}{u-s} E(X_um_n(X_u;u)|X_s).$$
 in the language of algebra $\AQH$  with  $\mmm_t=(m_0(x;t),m_1(x;t),\dots)$ becomes  \begin{equation}
  \label{harness*}
  \ppp_{s,t}\fff\mmm_t=\frac{u-t}{u-s} \fff\mmm_s+\frac{t-s}{u-s} \ppp_{s,u}\fff\mmm_u.
\end{equation}
Since each  polynomial $x m_n(x;t)$ can be written as the linear combination of $m_0(x;t)$, ..., $m_{n+1}(x;t)$,   one can find $\jjj_t\in\AQH$  such that
$\fff \mmm_t=\mmm_t \jjj_t$. Inserting this into \eqref{harness*}, we see that  martingale property eliminates $\ppp_{s,t}$ and  after left-multiplication by $\mmm_s^{-1}$ we get
$$
(u-s) \jjj_t=(u-t)\jjj_s+(t-s)\jjj_u.
$$
In particular, $\jjj_t$ depends linearly on $t$ and thus
$$\jjj_t=\jjj_0+ t (\jjj_1-\jjj_0).$$
This shows that for any martingale polynomials harness property implies that there exist $\uuu,\vvv\in\AQH$ such that
\begin{equation}
  \label{harness-M+}
  \fff\lp\mmm_t=\mmm_t\lp(\uuu+t\vvv).
\end{equation}
In our special case of  $\mmm_t=\ppp_{0,t}^{-1}$,
we get  $\uuu=\jjj_0=\fff \mmm_0=\fff\eee=\fff$.
This establishes \eqref{harness-M}, which  of course is equivalent to \eqref{harness}.
\subsubsection{Motivation for quadratic harness property \eqref{q-harness}}

 Suppose polynomial process $\{\ppp_{s,t}\}$ in the sense of Definition \ref{D1} arises from a  Markov process with polynomial conditional moments which is a harnesses and
  in addition  has  quadratic conditional variances \eqref{EQ:q-Var}.  Then from the previous discussion, \eqref{harness-M+} holds, and under mild technical assumptions, \cite[Theorem 2.3]{Bryc-Matysiak-Wesolowski-04} shows that   $\vvv,\uuu$   satisfy commutation equation which reduces  to \eqref{q-harness} when $\uuu=\fff$. This motivates condition (iii).

     In fact, it is known that under some additional  assumptions on the growth of moments,
       \eqref{harness} and \eqref{q-harness} imply  \eqref{EQ:LR} and \eqref{EQ:q-Var}, see
\cite[Section 4.1]{szablowski2014markov}; this equivalence is also implicit in the proof of \cite[Theorem
 2.3]{Bryc-Matysiak-Wesolowski-04} and is explicitly used in  \cite[page 1244]{Bryc-Wesolowski-08}.
However, this has no direct bearing on our paper, as  in this paper we simply adopt  the algebraic Definition \ref{D2}.

 \subsection{Infinitesimal generator}

 A polynomial process $\{\ppp_{s,t}\}$ with harness property \eqref{harness} is uniquely determined by $\vvv$.
Indeed, the $n$-th element of the sequence on the left hand side of \eqref{harness} is the $(n+1)$-th polynomial in $\ppp_{0,t}$ while the $n$-th element of the sequence on the right hand side of \eqref{harness} depends only on the first $n$ polynomials in $\ppp_{0,t}$.
In fact, one can check that
\begin{equation}\label{X2P}
\ppp_{0,t}=\sum_{k=0}^\infty(\fff+t\vvv)^k(\eee-\fff\ddd)\ddd^k
\end{equation}

\arxiv{\begin{proof}
   From \eqref{harness} we get
   \begin{equation}
     \label{Yk}
        \ppp_{0,t}\fff(\fff^k\ddd^k -\fff^{k+1}\ddd^{k+1})=(\fff+t\vvv)\ppp_{0,t}(\fff^k\ddd^k -\fff^{k+1}\ddd^{k+1})
   \end{equation}

Denote by $\uuu_k= \ppp_{0,t}(\fff^k\ddd^k -\fff^{k+1}\ddd^{k+1})$.   Multiplying \eqref{Yk} by $\ddd$ from the right, we get
$$
\uuu_{k+1}=(\fff+t\vvv)\uuu_k\ddd.
$$
Since $\uuu_0=\ppp_{0,t}(\eee-\fff\ddd)=\eee-\fff\ddd$ we get $\uuu_k=(\fff+t\vvv)^k(\eee-\fff\ddd)\ddd^k$. The telescoping series gives
$\ppp_{0,t}=\sum_{k=0}^\infty \uuu_k$.

\end{proof}

}
 Since the $k$-th
element of $\fff+t\vvv$  has degree $k+1$, it follows from \eqref{harness-M} that $\mmm_t$ is a rational function
of $t$.
Formula \eqref{m2p} shows that the left infinitesimal generator
\begin{equation}
  \label{LLL}
  \ggen_t=\lim_{h\to 0^+} \frac{1}{h}(\ppp_{t-h,t}-\eee), \; t>0
\end{equation}
is a well defined element of $\AQH$, and that
  \begin{equation}
    \label{A2M} \ggen_t\mmm_t= -\frac{\partial}{\partial t}\mmm_t
  \end{equation}  with differentiation defined componentwise.

Since $\ppp_{s,t}$ is continuous in $t$,   the right
infinitesimal generator exists and is given by the same expression. To see this, we compute $\lim_{h\to 0^+}
\frac{1}{h}(\ppp_{t,t+h}-\eee)$ on   $\mmm_t$. We have

\begin{multline*}
\lim_{h\to 0^+} \frac{1}{h}(\ppp_{t,t+h}\lp\mmm_t-\mmm_t)=
\lim_{h\to 0^+}  \ppp_{t,t+h}\lp\tfrac1h(\mmm_t-\mmm_{t+h}) \\
=\lim_{h\to 0^+}  \ppp_{t,t+h}\lp \lim_{h\to 0^+}   \tfrac1h(\mmm_t-\mmm_{t+h}) = -\frac{\partial}{\partial t}\mmm_t.
\end{multline*}

 The infinitesimal generator $\ggen_t$ determines $\ppp_{s,t}$ uniquely; for an algebraic proof see Proposition \ref{P-uni1}.
Clearly,  $\mmm_t=\eee-\int_0^t\aaa_s\mmm_s ds$.
%\comment{ $\mmm_t=\exp \left(-\int_0^t \ggen_s ds\right)$?}

The infinitesimal generator $\ggen_t$ and its companion operator $\gen_t:\calP\to\calP$
are the main objects of interest in this paper.

\begin{remark}\label{Rem:zeros}

We note that  $\ggen_t=(0,a_1(x;t),a_2(x;t),\dots)$ always starts with a 0, as from $\pp_{s,t}(1)=1$ it follows that
$\gen_t(1)=0$. It is clear that $a_n(x;t)$ is a polynomial in $x$ of degree at most $n$. Martingale property implies that
$a_1(x;t)=0$, as   $\gen_t(x)=0$. The infinitesimal generator considered as an element of $\AQH$ in the latter case starts with
two zeros, $\ggen_t=(0,0,a_2(x;t),a_3(x;t),\dots)$.
\end{remark}

\section{Basic properties of infinitesimal generators for  quadratic harnesses}

  Our first result introduces an auxiliary element $\hhh_t\in\AQH$ that is related to $\ggen_t$ by a   commutation equation.
\begin{theorem}
  \label{T-gen} Suppose that %
  $\{\ppp_{s,t}\in\AQH: 0\leq s\leq t\}$ is a quadratic harness  as in Definition \ref{D2} with generator $\ggen_t$.
For $t>0$, let   \begin{equation}
  \label{H2G} \hhh_t=\aaa_t\lp\fff-\fff\lp\aaa_t
\end{equation}
and denote $\ttt_t=\fff-t\hhh_t$. Then
   \begin{equation}
    \label{H}
 \hhh_t \lp\ttt_t- \gamma \ttt_t\lp \hhh_t
    = \eee+\theta \hhh_t+\eta \ttt_t+\tau \hhh_t^{  2}+\sigma \ttt_t^{  2}.
  \end{equation}

\end{theorem}
\begin{proof}
Differentiating \eqref{harness-M} and then using \eqref{A2M} we get
$$
-\fff\lp\ggen_t\lp\mmm_t=-\ggen_t\lp\mmm_t(\fff+t\vvv)+\mmm_t\lp\vvv.
$$
Multiplying this from the right by $\mmm_t^{-1}$ and using  \eqref{harness-M} to replace $\mmm_t(\fff+t\vvv)\mmm_t^{-1}$ by
$\fff$ we get
$$
-\fff\lp\ggen_t=-\ggen_t\lp\fff+\mmm_t\lp\vvv\mmm_t^{-1}.$$
Comparing this with \eqref{H2G} we see that
\begin{equation}
  \label{M2H}
  \hhh_t=\mmm_t\lp\vvv\mmm_t^{-1}.
\end{equation}
We now use this  and \eqref{harness-M}  in the definition of $\ttt_t =\fff-t\hhh_t=\mmm_t(\fff+t\vvv)\mmm_t^{-1}-t \hhh_t$. Using \eqref{M2H} we get
%After replacing $\fff$ by $\mmm_t(\fff+t\vvv)\mmm_t^{-1}$  we get
$$
\ttt_t=\mmm_t\lp\fff\mmm_t^{-1}.
$$
To derive equation \eqref{H} we now multiply  \eqref{q-harness}  by $\mmm_t$ from the left and by $\mmm_t^{-1}$ from the
right.
\end{proof}

\begin{remark}\label{Rem:zerosH}
As observed in  Remark \ref{Rem:zeros}
the $n$-th element of $\ggen_t$ is a polynomial of degree at most $n$. Thus
writing $\hhh_t=(h_0(x),h_1(x),\dots)$,  from \eqref{H2G} we see that  $h_n$ is of degree at most $n+1$.
Martingale property   implies that $h_0(x)=0$.
\end{remark}

We will need several  uniqueness results that follow from the more detailed  elementwise  analysis of the sequences of
polynomials. We first show that the generator determines uniquely the polynomial process, at least under the harness
\eqref{harness}  condition. %(For homogeneous polynomial processes \eqref{A2P-hom} is a  stronger version of this result.)

%\comment{Is this used? Needed?}
 \begin{proposition}  \label{P-uni1} %
A polynomial process    $\{\ppp_{s,t}:0\leq s<t\}$ with harness property \eqref{harness}  is determined uniquely  by its  generator $\aaa_t$.
\end{proposition}

\begin{proof}

Fix $s<t$. Combining \eqref{m2p} with \eqref{harness-M} we get
\begin{multline*}
  \ppp_{s,t}\lp\fff-\fff\lp\ppp_{s,t}= \mmm_s\mmm_t^{-1}\fff-\fff\mmm_s\mmm_t^{-1} \\
  =\mmm_s(\uuu+t\vvv)\mmm_t^{-1}-\mmm_s(\uuu+s\vvv)\mmm_t^{-1} = (t-s)\mmm_s\mmm_t^{-1}\left(\mmm_t\vvv\mmm_t^{-1}\right).
\end{multline*}
Therefore, from \eqref{M2H} we get
$$
\ppp_{s,t}\lp \fff=\fff\lp\ppp_{s,t}+(t-s)\ppp_{s,t}\lp\hhh_t.
$$

With $s=0$, this implies
$$
\mmm_t\fff=(\fff-t \hhh_t)\mmm_t.
$$

This equation is similar to equation \eqref{harness} and again uniqueness follows from consideration of the degrees of the polynomials.
The  solution is
$$
\mmm_t=\sum_{k=0}^\infty (\fff-t\hhh_t)^k(\eee-\fff\ddd)\ddd^k,
$$
compare\eqref{X2P}. Thus  $\mmm_t$ is determined uniquely, and  \eqref{m2p} shows that operators $\ppp_{s,t}$ are uniquely determined.
 Since $\hhh_t$ is expressed  in terms of $\ggen_t$ by \eqref{H2G}, this ends the proof.

\end{proof}

Next, we show that  $\hhh_t=(h_0,h_1,\dots)$   is uniquely determined by the commutation equation \eqref{H} with the ``initial condition'' $h_0=0$. From the proof of Proposition \ref{P-uni1} it therefore follows that
the entire  quadratic harness $\{\ppp_{s,t}\}$ as well as its generator are also uniquely determined by  \eqref{H}.

\begin{proposition}  \label{P-uni2} %\comment{Revised statement and proof!}
If $\sigma,\tau\geq0$,  %$\gamma\leq 1+2\sqrt{\sigma\tau}$
and $\sigma\tau\ne 1$ then equation \eqref{H}
has a unique solution among $\hhh_t\in\AQH$ such that $h_0(x)=0$.
\end{proposition}
\begin{proof}
Eliminating  $\ttt_t=\fff-t\hhh_t$ from \eqref{H} we can rewrite it into the following equivalent form.
\begin{equation}\label{GlownyWzorek}
\hhh_t\lp \fff-\gamma\fff\lp \hhh_t=\eee+\theta \hhh_t+\eta(\fff-t\hhh_t)+\tau \hhh_t^{\lp 2}+\sigma (\fff-t\hhh_t)^{\lp2}+(1-\gamma)t\hhh_t^{\lp 2}.
\end{equation}

Write $\hhh_t=\hhh=(h_n(x))_{n=0,1,\dots}$ for a fixed $t>0$, with $h_0=0$.
We will simultaneously prove that for $n\geq1$, polynomial $h_n(x)$ is uniquely determined and that its  degree is at most $n+1$.
The proof is by induction. With $h_0=0$ it is clear that the thesis holds true for $n=0$. We therefore assume that $h_0,h_1,\dots,h_n$ are given polynomials of degrees at most $1,2,\dots,n+1$, respectively.

Looking at the $n$-th element   of   \eqref{GlownyWzorek} for $n\geq 0$,
we get the following equation for $h_{n+1}(x)$:
 \begin{multline}\label{h-recurrence-v0}
\left(1+\sigma t- [h_n]_{n+1}(\sigma t^2+(1-\gamma)t+\tau) \right)h_{n+1}(x)
\\=x^n+\eta x^{n+1}+\sigma x^{n+2}+(\theta-t\eta)h_n(x)\\+(\gamma-\sigma t)xh_n(x)+(\sigma t^2+(1-\gamma)t+\tau) \sum_{j=0}^{n} [h_n]_j h_j(x)\,.
\end{multline}
The degree of the polynomial on the right hand side is at most $n+2$, as the highest degree term on the right hand side is $\left(\sigma+(\gamma-\sigma t)[h_n]_{n+1}\right)x^{n+2}$.
In order to complete the proof, we only need to verify that for all $n\geq 0$, the coefficient $1+\sigma t- [h_n]_{n+1}(\sigma t^2+(1-\gamma)t+\tau) $ on the left hand side of \eqref{h-recurrence-v0} does not vanish.

%Equation \eqref{h-recurrence-v0} takes a simpler form when $\sigma=0$, and the proof for this case needs some modifications. In the remainder of the proof we assume that $\sigma> 0$.

 We   need to consider  separately two cases.
\begin{description}
\item [Case $\gamma +\sigma\tau\ne 0$] We proceed by contradiction.
Suppose  that for some $n\geq 0$ the coefficient at $h_{n+1}(x)$ on the left hand side of \eqref{h-recurrence-v0} is  $0$. Since $1+\sigma t>0$, this implies that $\sigma t^2+(1-\gamma)t+\tau$ cannot be zero. So  we get
$$
[h_n]_{n+1}=\frac{1+\sigma t}{\sigma t^2+(1-\gamma)t+\tau}.
$$
We now use this value to compute  the coefficient at $x^{n+2}$ on the right hand side of  \eqref{h-recurrence-v0}. We get
$$\sigma+(\gamma-\sigma t)[h_n]_{n+1}=\frac{\gamma+\sigma\tau}{\sigma t^2+(1-\gamma)t+\tau}\ne 0.$$
Since the left hand side of  \eqref{h-recurrence-v0} is $0$, and the degree of the right hand side of  \eqref{h-recurrence-v0} is $n+2$, this is a contradiction.

This shows that the coefficient of $h_{n+1}(x)$ on the left hand side of \eqref{h-recurrence-v0} is  non-zero for all
 $n\geq0$. So  each polynomial $h_{n+1}(x)$ is determined uniquely and has degree at most $n+2$ for all $n\geq 0$.

\item [Case $\gamma+\sigma\tau=0$]
In this case  $ \sigma t^2+(1-\gamma)t+\tau  = (1+\sigma t)(\tau+t) $.    Comparing the  coefficients at $x^{n+2}$ on both sides of \eqref{h-recurrence-v0}  we get
\begin{equation}
  \label{ComonFact}
  (1+\sigma t)(1-[h_n]_{n+1}(\tau+t)) [h_{n+1}]_{n+2}= \sigma (1-[h_n]_{n+1}(\tau+t)).
\end{equation}

Since $h_0(x)=0$, this gives  $[h_1]_{2}= \sigma/(1+\sigma t)$. So for $n=1$  we get
 $1-[h_n]_{n+1}(\tau+t)=(1-\sigma\tau)/(1+\sigma t)\ne 0$.  Dividing both sides of \eqref{ComonFact} by  this expression we
 get recursively   $[h_n]_{n+1}=\sigma/(1+\sigma t)$ for all $n\geq 1$. Therefore, for $n\geq 1$, the left hand side of
  \eqref{h-recurrence-v0} simplifies  to $(1-\sigma \tau) h_{n+1}(x)$. Using again $1-\sigma\tau\ne0$, this shows that polynomial
  $h_{n+1}$ is determined uniquely and its degree is at most $n+2$. Of course, \eqref{h-recurrence-v0} determines
   $h_1(x)$   uniquely, too,  as $h_0(x)=0$.
\end{description}

\end{proof}

Our main result is the identification of the infinitesimal generator for quadratic harnesses with $\gamma=-\sigma\tau$. Such processes were called ``free quadratic harnesses" in
 \cite[Section 4.1]{Bryc-Matysiak-Wesolowski-04}.   Markov processes with free quadratic  harness property  were constructed in  \cite{Bryc-Matysiak-Wesolowski-05} but the construction  required more restrictions on the parameters than what we impose here.

%\comment{re-edited}
In this section we represent the infinitesimal generators using the auxiliary power series $\varphi_t(\ddd)$ with $\ddd$ defined by
\eqref{def-ddd}. In Section \ref{Sect:IR} we will use the results of this section to derive the  integral representation
\eqref{gen-free} for the  operator $\gen_t$ under
 a more restricted range of parameters $\eta,\theta,\sigma,\tau$.

For a formal power series  $\varphi(z)=\sum_{k=0}^\infty c_kz^k$ %
we shall
write $\varphi(\ddd)$ for the series  $\sum_kc_k \ddd^k$.
   We note that since  the sequence $\ddd^k$ begins with $k$ zeros, $\sum_kc_k \ddd^k$ is a sequence of finite sums:
$$\varphi(\ddd)=(c_0,c_0x+c_1,c_0x^2+c_1x+c_2,\dots,\sum_{j=0}^n c_j x^{n-j},\dots).$$
So  $\varphi(\ddd) $ is a  well defined element of $\AQH$.

  We will also need $\ddd_1=\sum_{k=0}^\infty \fff^k\ddd^{k+1} =(0,1,2x,3x^2,\dots)$ which represents the derivative.

\begin{theorem}\label{T:FH}
Fix  $\eta,\theta\in\RR$ and  $\sigma,\tau\geq 0$ such that $\sigma\tau\ne 1$.
Then the infinitesimal generator of the quadratic harness with the above parameters and with $\gamma=-\sigma\tau$ is given by
\begin{equation}\label{gen-free-0}
  \ggen_t=\frac{1}{1+\sigma t}(\eee+\eta\fff+\sigma\fff^2)\ddd_1\varphi_t(\ddd)\ddd, \; t>0,\end{equation}
where
$\varphi_t(z)=\sum_{k=1}^\infty  c_k(t) z^{k-1}$ for small enough $z$  solves the    quadratic equation \begin{equation}
\label{phi-eqtn} (z^2+\eta z +\sigma ) (t+\tau )
   \varphi_t^2+
   ((\theta-t  \eta)z
   -2 t \sigma -\sigma  \tau
   -1) \varphi_t +t \sigma +1=0
\end{equation}
  and the solution is chosen so that $\varphi_t(0)=1$.

  (For $\sigma\tau<1$ this solution is written explicitly in formula \eqref{varphi-sol} below.)
\end{theorem}

 We note that for each fixed $t$ equation \eqref{phi-eqtn} has two real roots for $z$
close enough to $0$. As  $\varphi_t(z)$ we choose the smaller root when $\sigma\tau<1$ and the larger root when $\sigma\tau>1$.
For $z=0$, equation \eqref{phi-eqtn} becomes
$$
\sigma (t+\tau)\varphi_t^2(0)- (1+\sigma \tau + 2 t \sigma)\varphi_t(0)+1+t\sigma=0,
$$
so  this procedure ensures that   $\varphi_t(0)=1$.

%\comment{Uwaga: szereg $\varphi_t(z)$ ma promien zbieznosci zalezny od $t$}

\section{Proof of Theorem \ref{T:FH}} \label{sect-Proof}
The plan of proof is to   solve equation \eqref{H} %with $\gamma=-\sigma\tau$
for $\hhh_t$, and then to use equation \eqref{H2G}  to determine $\ggen_t$.

\subsection{Part I of proof: solution of equation \eqref{H} when $\gamma=-\sigma\tau$}

Equation \eqref{H} takes the form
\begin{multline*}
\hhh_t \lp \fff -t\hhh_t^{\lp2}+\sigma\tau\fff\lp\hhh_t-\sigma\tau t\hhh_t^{\lp 2} \\ =\eee+\theta \hhh_t+\eta(\fff-t\hhh_t)+\tau\hhh_t^{\lp2}+\sigma (\fff^{\lp2}-t\hhh_t\lp\fff-t\fff\lp\hhh_t+t^2\hhh_t^{\lp2}).
\end{multline*}
So after simplifications, the equation to solve for the unknown $ \hhh_t$ is
\begin{equation}\label{freeH2}
(1+\sigma t)\hhh_t\lp\fff =\eee+\eta\fff+\sigma\fff^{\lp2}+(\theta-\eta t)\hhh_t-\sigma(t+\tau)\fff\lp\hhh_t+(t+\tau)(1+\sigma t)\hhh_t^{\lp2}\,.
\end{equation}

\begin{lemma}  %
The solution of \eqref{freeH2}   with  the initial element $h_0=0$  is
  \begin{equation}
    \label{phi2h}
    \hhh_t=\frac{1}{1+\sigma t}(\eee+\eta\fff+\sigma\fff^2)\varphi_t(\ddd)\ddd,
  \end{equation}
  where $\varphi_t$ satisfies equation \eqref{phi-eqtn} and $\varphi_t(0)=1$.
\end{lemma}
\begin{proof}
Since $t>0$ is fixed, we suppress the dependence on $t$ and we use Remark \ref{Rem:zerosH} to write
$\hhh_t=\hhh=(0,h_1(x),\dots)$. From \eqref{freeH2} we read out that $h_1(x)=\frac{1}{1+\sigma t}(1+\eta x+\sigma x^2)$.

From Proposition \ref{P-uni2} we see   that \eqref{freeH2} has a unique solution.
In view of uniqueness, we seek the solution in a special form
\begin{equation}
  \label{specform}
\hhh=\frac{1}{1+\sigma t}(\eee+\eta\fff+\sigma\fff^{\lp2})\lp\sum_{k=1}^\infty c_k\ddd^{\lp k}
\end{equation}
with $c_1=1$ and $c_k=c_k(t)\in\RR$.
Note that
\begin{multline*}
\hhh^2=\frac{1}{(1+\sigma t)^2}(\eee+\eta\fff+\sigma\fff^{\lp2})\lp\sum_{k=1}^\infty c_k\ddd^{\lp k}\lp(\eee+\eta\fff+\sigma\fff^{\lp2})\lp\sum_{j=1}^\infty c_j\ddd^{\lp j}
\\=
\frac{1}{(1+\sigma t)^2}(\eee+\eta\fff+\sigma\fff^{\lp2})\lp\sum_{k=1}^\infty c_k\ddd^{\lp k} \lp\sum_{j=1}^\infty c_j\ddd^{\lp j}
\\
+\frac{\eta}{(1+\sigma t)^2}(\eee+\eta\fff+\sigma\fff^{\lp2})\lp \sum_{k=1}^\infty c_k\ddd^{\lp k-1} \lp\sum_{j=1}^\infty c_j\ddd^{\lp j}\\
+\frac{\sigma}{(1+\sigma t)^2}(\fff+\eta\fff^2+\sigma\fff^{\lp3})\lp \sum_{j=1}^\infty c_j\ddd^{\lp j}
+\frac{\sigma}{(1+\sigma t)^2}(\eee+\eta\fff+\sigma\fff^{\lp2})\lp  \sum_{k=2}^\infty c_k\ddd^{\lp k-2}\lp \sum_{j=1}^\infty c_k\ddd^{\lp j}\,.
\end{multline*}

Inserting this into \eqref{freeH2} we get

\begin{multline*}%
\eee+\eta\fff + \sigma\fff^{\lp2}+ \left(\eee+\eta\fff + \sigma\fff^{\lp2}\right)\sum_{k=2}^\infty c_{k }\ddd^{\lp k-1}
=\eee+\eta\fff + \sigma\fff^{\lp2}\\+\frac{\theta-t\eta}{1+\sigma t}(\eee+\eta\fff+\sigma\fff^{\lp2})\lp\sum_{k=1}^\infty c_k\ddd^{\lp k} -\frac{\sigma(t+\tau)}{1+\sigma t}(\fff+\eta\fff^{\lp2}+\sigma\fff^{\lp3})\lp\sum_{k=1}^\infty c_k\ddd^{\lp k}
\\
+ \frac{t+\tau}{1+\sigma t}\Big[(\eee+\eta\fff+\sigma\fff^{\lp2})\lp\sum_{k=1}^\infty c_k\ddd^{\lp k} \lp\sum_{j=1}^\infty c_j\ddd^{\lp j}
+\eta(\eee+\eta\fff+\sigma\fff^{\lp2})\lp \sum_{k=1}^\infty c_k\ddd^{\lp k-1} \lp\sum_{j=1}^\infty c_j\ddd^{\lp j}\\
+\sigma(\fff+\eta\fff^2+\sigma\fff^{\lp3})\lp \sum_{j=1}^\infty c_j\ddd^{\lp j}
+\sigma(\eee+\eta\fff+\sigma\fff^{\lp2})\lp  \sum_{k=2}^\infty c_k\ddd^{\lp k-2}\lp \sum_{j=1}^\infty c_j\ddd^{\lp j}\Big]\,.
\end{multline*}
The terms with $\left(\fff+\eta\fff^2 + \sigma\fff^{\lp3}\right)$ cancel out, so $\left(\eee+\eta\fff + \sigma\fff^{\lp2}\right)$ factors out.
We further restrict our search for the solution by requiring that %
the remaining factors   match, i.e.
\begin{multline*}%
   \sum_{k=1}^\infty c_{k+1 }\ddd^{\lp k }
=\frac{\theta-t\eta}{1+\sigma }\sum_{k=1}^\infty c_k\ddd^{\lp k}  \\
+ \frac{t+\tau}{1+\sigma t}\Big(\sum_{k=1}^\infty c_k\ddd^{\lp k} \lp\sum_{j=1}^\infty c_j\ddd^{\lp j}
+\eta  \sum_{k=1}^\infty c_k\ddd^{\lp k-1}  \lp\sum_{j=1}^\infty c_j\ddd^{\lp j}
+\sigma   \sum_{k=2}^\infty c_k\ddd^{\lp k-2}\lp \sum_{j=1}^\infty c_j\ddd^{\lp j}\Big)\,.
\end{multline*}
Collecting the coefficients at the powers of $\ddd$ we get
\begin{multline*}%
   \sum_{k=1}^\infty c_{k+1 }\ddd^{\lp k }
=\frac{\theta-t\eta}{1+\sigma t}\sum_{k=1}^\infty c_k\ddd^{\lp k}  \\
+ \frac{t+\tau}{1+\sigma t}\Big( \sum_{k=2}^\infty \sum_{j=1}^{k-1} c_j  c_{k-j}\ddd^{\lp k}
+\eta  \sum_{k=1}^\infty  \sum_{j=0}^{k-1} c_{j+1}  c_{k-j}\ddd^{\lp k}
+\sigma   \sum_{k=1}^\infty  \sum_{j=0}^{k-1} c_{j+2}  c_{k-j}\ddd^{\lp k}\Big)\,.
\end{multline*}

We now compare the coefficients at the powers of $\ddd$. Since $\sigma\tau\ne 1$, for $k=1$ we get
$$
c_2=\frac{\theta-\eta t}{1+\sigma t}+\frac{t+\tau}{1+\sigma t}\eta+\frac{t+\tau}{1+\sigma t}\sigma c_2\,.
$$
So
$c_2=(\theta+\eta\tau)/(1-\sigma\tau)=\beta$ (say).

For $k\geq 2$, we have the recurrence
\begin{equation}
  \label{c-recursion}
  c_{k+1}=\frac{\theta-\eta t}{1+\sigma t}c_k+\frac{t+\tau}{1+\sigma t}\Big( \sum_{j=1}^{k-1} c_j  c_{k-j}+\eta \sum_{j=0}^{k-1} c_{j+1}  c_{k-j}+\sigma  \sum_{j=0}^{k-1} c_{j+2}  c_{k-j}\Big)\,.
\end{equation}

We solve this recurrence by  the method of  generating functions.  One can proceed here with a formal
power series, and then invoke uniqueness to verify that the power series has positive radius of convergence. Or one can use an
a'priori bound from Lemma \ref{L-growth} below and restrict the argument of the generating function to small enough $|z|$. Note
that in this argument, $t$ is fixed.

Let $\varphi(z)=\sum_{k=1}^\infty c_kz^{k-1}$. Then
\begin{multline*}%
\varphi(z)=1+\beta z+\sum_{k=2}^\infty c_{k+1}z^k\\
= \frac{\theta-\eta t}{1+\sigma t} z(\varphi(z)-1)+
\frac{t+\tau}{1+\sigma t} z^2\varphi^2(z)+\eta z\frac{t+\tau}{1+\sigma t} (\varphi^2(z)-1) \\
+\sigma \frac{t+\tau}{1+\sigma t} (\varphi(z)(\varphi(z)-1) -\beta z).
\end{multline*}
This gives     quadratic equation \eqref{phi-eqtn} for $\varphi=\varphi_t$.

\end{proof}

 The following technical lemma assures that the series  $\varphi(z)=\sum_{k=1}^\infty c_kz^{k-1}$ converges for all small enough $|z|$.
\begin{lemma}\label{L-growth} Suppose $\sigma\tau\ne1$ and $\{c_k\}$ is a solution of recursion \eqref{c-recursion} with $c_1=1$ for a fixed $t$. Then for every $p>1$ there exists a constant $M$ such that
\begin{equation}\label{c-bound}
|c_k|\leq \frac{M^{k-2}}{k^p} \mbox{ for all  $k\geq 3$.}
\end{equation}

\end{lemma}
\begin{proof} We will prove the case $p=2$ only, as this suffices to justify the convergence of the series.

Solving \eqref{c-recursion} for $c_{k+1}$, which appears also in one term on the right hand side of \eqref{c-recursion}, we get

$$
\frac{|1-\sigma\tau|}{1+\sigma t} |c_{k+1}|\leq
\frac{|\theta-\eta t|}{1+\sigma t}|c_k|+\frac{t+\tau}{1+\sigma t}\Big( \sum_{j=1}^{k-1} |c_j c_{k-j}|+|\eta| \sum_{j=0}^{k-1} |c_{j+1} c_{k-j}|+\sigma  \sum_{j=0}^{k-2} |c_{j+2} c_{k-j}|\Big)
$$
Since we are not going to keep track of the constants, we simplify this as

\arxiv{
$$ |c_{k+1}|\leq
\frac{|\theta-\eta t|}{1-\sigma\tau}|c_k|+\frac{t+\tau}{|1-\sigma\tau|}\Big( \sum_{j=1}^{k-1} |c_j c_{k-j}|+|\eta| \sum_{j=0}^{k-1} |c_{j+1} c_{k-j}|+\sigma  \sum_{j=0}^{k-2} |c_{j+2} c_{k-j}|\Big)
$$
}
\begin{equation}
  \label{c-bound2}
   |c_{n+1}|\leq
A|c_n|+B\Big( \sum_{k=1}^{n-1} |c_kc_{n-k}|+ \sum_{k=0}^{n-1} |c_{k+1} c_{n-k}|+   \sum_{k=0}^{n-2} |c_{k+2} c_{n-k}|\Big)
\end{equation}
with $A=\frac{|\theta-\eta t|}{|1-\sigma\tau|}$ and $B=\frac{t+\tau}{|1-\sigma\tau|}(1+|\eta|+\sigma)$. Here,  $n\geq 2$.

We now choose $M\geq 1$ large enough so that \eqref{c-bound} holds for $k=3,4,5,6$,
%$$|c_3|\leq \frac{M}{3^2},\; |c_4|\leq \frac{M^2}{4^2},\;  |c_5|\leq \frac{M^3}{5^2},\;  |c_6|\leq
%\frac{M^4}{6^2},$$
and we also require that
\begin{equation}
  \label{M-choice}
 4A+ (24|c_2|+175)B\leq M.
\end{equation}
  %With the above $M$, we see that  \eqref{c-bound} holds for $k=3,\dots, 6$.
  We now proceed by induction and assume
that \eqref{c-bound} holds for all indices $k$ between $3$ and $n$ for some $n\geq 6$.

To complete the induction step we provide bounds for the sums on the right hand side of \eqref{c-bound2}. The first sum is
handled as follows:
$$
\sum_{k=1}^{n-1} |c_{k}c_{n-k}|=2|c_1c_{n -1}|+2|c_2 c_{n-2}|+\sum_{k=3}^{n-3}|c_{k} c_{n-k}|.
$$
We   apply the induction bound \eqref{c-bound} to $c_3,\dots,c_{n-1}$. Noting that  $c_1=1$ we get
\begin{multline*}
\sum_{k=1}^{n-1} |c_kc_{n-k}|\leq 2\frac{M^{n-3}}{(n-1)^2}+2|c_2| \frac{M^{n-4}}{(n-2)^2}+M^{n-4}\sum_{k=3}^{n-3}\frac{1}{k^2(n-k)^2}\\
\leq 8\frac{M^{n-2}}{(n+1)^2}+8|c_2|\frac{M^{n-2}}{(n+1)^2}+M^{n-2}\sum_{k=3}^{[n/2]}\frac{1}{k^2(n-k)^2}
+M^{n-2}\sum_{k=[n/2]+1}^{n-3}\frac{1}{k^2(n-k)^2}\\
\leq 8\frac{M^{n-2}}{(n+1)^2}+8|c_2|\frac{M^{n-2}}{(n+1)^2}+\frac{4 M^{n-2}}{n^2}\sum_{k=3}^{[n/2]}\frac{1}{k^2}+
+\frac{4M^{n-2}}{n^2}\sum_{k=[n/2]+1}^{n-2}\frac{1}{(n-k)^2}
\\ < \frac{M^{n-2}}{(n+1)^2}\left(
8+8|c_2|+16\pi^2/3 \right)<(61+8|c_2|)\frac{M^{n-2}}{(n+1)^2} .
\end{multline*}
Here we used several times the  bound $(n+1)/(n-2)\leq 2$ for $n\geq 6$, inequality $M\geq 1$, and we estimated two finite sums
by the infinite series $\sum_{k=1}^\infty 1/k^2=\pi^2/6$.

We proceed similarly with the second sum:
\begin{multline*}
\sum_{k=0}^{n-1} |c_{k+1}c_{n-k}|=2|c_1c_{n }|+2|c_2 c_{n-1}|+\sum_{k=2}^{n-3}|c_{k+1} c_{n-k}| \\ \leq 2
\frac{M^{n-2}}{n^2}+2|c_2|\frac{M^{n-3}}{(n-1)^2}+\frac{4M^{n-3}}{n^2}\sum_{k=1}^{[n/2]}\frac{1}{k^2}+
\frac{4M^{n-3}}{n^2}\sum_{k=[n/2]+1}^{n-3}\frac{1}{(n-k)^2}
\\ \leq \frac{M^{n-2}}{(n+1)^2}\left(8+8 |c_2|+16\pi^2/3\right)\leq
\frac{M^{n-2}}{(n+1)^2}\left(61+8 |c_2|\right).
\end{multline*}

Finally, we bound the third sum on the right hand side of \eqref{c-bound2}. We get
\begin{multline*}
\sum_{k=0}^{n-2} |c_{k+2}c_{n-k}|= 2|c_2 c_{n}|+\sum_{k=1}^{n-3}|c_{k+2} c_{n-k}| \\ \leq
2|c_2|\frac{M^{n-2}}{n^2}+\frac{4M^{n-2}}{n^2}\sum_{k=1}^{[n/2]}\frac{1}{k^2}+
\frac{4M^{n-2}}{n^2}\sum_{k=[n/2]+1}^{n-3}\frac{1}{(n-k)^2}
\\ \leq \frac{M^{n-2}}{(n+1)^2}\left( 8 |c_2|+16\pi^2/3\right)\leq
\frac{M^{n-2}}{(n+1)^2}\left(8 |c_2|+53\right).
\end{multline*}

Combining   these bounds together and using  the induction assumption to the first term on the right hand side of
\eqref{c-bound2} we get
 $$|c_{n+1}|\leq  \frac{M^{n-2}}{(n+1)^2}\left(4A+ (24|c_2|+175)B\right) $$
 In view of \eqref{M-choice}, this shows that
 $|c_{n+1}|\leq  \frac{M^{n-1}}{(n+1)^2}$, thus completing the proof of \eqref{c-bound} by induction.
\end{proof}

\subsection{Part II of proof: solution of equation \eqref{H2G}}

We now use \eqref{phi2h} to determine $\ggen_t$.

Since $\eee-\fff\ddd=(1,0,0,\dots)$,
from Remark \ref{Rem:zeros}  it follows that $\ggen_t\lp\fff\lp\ddd= \ggen_t$. Therefore, multiplying \eqref{H2G} by $\ddd$ from the right we get
$$
\ggen_t=\fff\lp\ggen_t\lp\ddd+ \hhh_t\lp\ddd.
$$
Iterating this, we get
\begin{equation}
  \label{H2G-sol}
 \ggen_t=\sum_{k=0}^\infty \fff^k \lp\hhh_t \lp\ddd^{k+1},
\end{equation}
which is well defined as the series consists of finite sums elementwise. Since $\hhh_t$ is given by \eqref{phi2h}, we get
\begin{multline} \label{AAA-sum}
\ggen_t=\frac{1}{1+\sigma t}\sum_{k=0}^\infty \fff^k \lp(\eee+\eta\fff+\sigma\fff^2) \varphi_t(\ddd)\ddd \lp\ddd^{k+1}.
\\=\frac{1}{1+\sigma t} \lp(\eee+\eta\fff+\sigma\fff^2) \sum_{k=0}^\infty \fff^k\lp\ddd^{k+1}\varphi_t(\ddd)\ddd.
\end{multline}
We now note that
 \begin{equation}
   \label{series4D1}
 \sum_{k=0}^\infty \fff^k \ddd^{k+1}=  \ddd_1.
 \end{equation}
 (This can be seen either by examining each element of the sequence, or by solving the equation $\ddd_1\fff-\fff\ddd_1=\eee$, which is just a product formula for the derivative, by the previous technique.
 The latter equation  is of course of the same form as \eqref{H2G}.)

Replacing the series in \eqref{AAA-sum} by the right hand side of \eqref{series4D1} we get \eqref{gen-free-0}.
This ends the proof of Theorem \ref{T:FH}.

\section{Integral representation}\label{Sect:IR}
In this section it is more convenient to use  the  linear operators on $\calP$ instead of the sequences of polynomials. We
assume that the polynomial process $\{\pp_{s,t}:0\leq s\leq t\}$ corresponds to a quadratic harness from Definition \ref{D2},
and as before we use the parameters $\eta,\theta,\sigma,\tau$ and $\gamma$ to describe the quadratic harness. Infinitesimal
generators of several quadratic harnesses, all different than those in Theorem \ref{T:FH}, have been studied in this language
by several authors.

For    quadratic harnesses with parameters $\eta=\sigma=0$ and    $\gamma=q\in(-1,1)$, according to
\cite{Bryc-Wesolowski-2013-gener}, the infinitesimal generator  $\gen_t$ acting on a polynomial $f$ is
\begin{equation}
  \label{gen-q-meixner}
  \gen_t(f)(x)=\int_\RR \frac{\partial }{\partial x} \left( \frac{f(y)-f(x)}{y-x}\right)\nu_{x,t}(dy),
\end{equation}
where $\nu_{x,t}(dy)$ is a uniquely determined probability measure. By inspecting the recurrences for the orthogonal
polynomials $\{Q_n\}$ and $\{W_n\}$, from \cite[Theorem 1.1(ii)]{Bryc-Wesolowski-2013-gener} one can read out that for
$q^2t\geq (1+q)\tau$   probability measure $\nu_{x,t}$ can be expressed in terms of  the transition probabilities
$P_{s,t}(x,dy)$    of the Markov process by the formula $\nu_{x,t}(dy)=P_{tq^2-(1+q)\tau,t}(\theta+qx,dy)$. In this form, the
formula coincides with   Anshelevich \cite[Corollary 22]{Anshelevich:2011} who considered the case $\eta=\theta=\tau=\sigma=0$
and $\gamma=q\in[-1,1]$. (However, the domain of the generator  in \cite{Anshelevich:2011} is much larger than the
polynomials.) Earlier results in Refs.
 \cite[page 392]{Biane:1998b},  \cite[Example 4.9]{BKS97}, and   \cite{Bryc-free-gen:2009}
dealt with   infinitesimal generators for  quadratic harnesses such that $\sigma=\eta=\gamma=0$.

 The following result gives an  explicit formula for the infinitesimal generator of
 the evolution corresponding to the ``free quadratic harness''
 in the integral form similar to \eqref{gen-q-meixner}. The main new feature   is the presence of an extra quadratic
  factor in front of the integral in expression  \eqref{gen-free} for the infinitesimal  generator.
 Denote
\begin{equation}\label{alphabeta}
\alpha=\frac{\eta+\theta\sigma}{1-\sigma\tau},\ \beta=\frac{\eta\tau+\theta}{1-\sigma\tau}.
\end{equation}
\begin{theorem}\label{T1} Fix  $\sigma,\tau\ge0$ such that
$\sigma\tau<1$ and $\eta,\theta\in\RR$ such that $1+\alpha\beta>0$.  Let  $\gamma=-\sigma\tau$ and let $\varphi_t$ be a continuous solution of \eqref{phi-eqtn} with $\varphi_t(0)=1$.

Then $\varphi_t$   is a moment generating function of a unique probability measure $\nu_t$ and the generator of the quadratic harness with the above parameters on $p\in\calP$ is
\begin{equation}\label{gen-free}
\gen_t(p)(x)=\frac{1+\eta x+\sigma x^2}{1+\sigma t}\int  \frac{\partial }{\partial x}\left(\frac{p(y)-p(x)}{y-x}\right)\nu_{t}(dy), \; t>0.
\end{equation}

 \end{theorem}

\begin{proof}
Since $\sigma\tau<1$, the solution of quadratic equation \eqref{phi-eqtn} with $\varphi_t(0)=1$ is
\begin{multline}\label{varphi-sol}
\varphi_t(z)=\frac{z(t  \eta - \theta) +2 t \sigma +\sigma  \tau +1}{2 (t+\tau )(z^2 +z\eta +\sigma  ) } \\-\frac{ \sqrt{
(z (t  \eta   - \theta) +2 t \sigma +\sigma  \tau +1)^2-4 (z^2 +z\eta +\sigma ) (1+t \sigma ) (t+\tau )}}{2(t+\tau ) (z^2 +z\eta +\sigma ) }.
\end{multline}
(We omit the  re-write for the case $\sigma=\eta=0$.)

\arxiv{ To avoid issues with $\sigma=\eta=0$, it is better to rewrite   \eqref{varphi-sol}  as
$$
\varphi_t(z)=\frac{2 (1+t \sigma )  }{ A+\sqrt{A^2-4 B}}.
$$
with $A=z(t  \eta - \theta) +2 t \sigma +\sigma  \tau +1$ and $B=(z^2 +z\eta +\sigma ) (1+t \sigma ) (t+\tau )$ }

 We will identify $\nu_t$ through its Cauchy-Stieltjes transform
    $$G_{\nu_t}(z)= \int\frac{1}{z-x}\nu_t(dx)$$
 which in our case will be well defined for all real $z$ large enough.

To this end we compute
\begin{multline}\label{duze-G0}
\varphi_t(1/z)/z= \frac{(1+\sigma\tau+2\sigma t)z+t\eta-\theta}{2(t+\tau)(\sigma z^2+\eta z+1)}\\
-\frac{\sqrt{\left[(1-\sigma\tau)z-(\alpha+\sigma\beta)t-\beta-\alpha\tau\right]^2-4(1+\sigma t)(t+\tau)(1+\alpha\beta)}}{2(t+\tau)(\sigma z^2+\eta z+1)}.
\end{multline}
Under our assumptions, the above expression is well defined for real large enough $z\in\RR$.

Expression \eqref{duze-G0}  coincides with  the Cauchy-Stieltjes transform in \cite[Proposition 2.3]{Saitoh-Yoshida01},    with their parameters
$$c_{SY}=\frac{1-\sigma\tau}{1+\sigma t}, \; \alpha_{SY}=\frac{\eta  \tau +\theta }{1-\sigma  \tau },\; a_{SY}=\frac{2 \eta  \tau +\theta  \sigma  \tau +\theta +t (\eta  \sigma
   \tau +\eta +2 \theta  \sigma )}{(\sigma  \tau -1)^2}$$
    and $$
   b_{SY}=\frac{(\sigma  t+1) (t+\tau ) \left(\eta ^2 \tau +\eta  \theta
   (\sigma  \tau +1)+\theta ^2 \sigma +(1-\sigma  \tau
   )^2\right)}{(1-\sigma  \tau )^3}.
   $$
   (We added subscript "SY"  to avoid confusion with our use of $\alpha$ in \eqref{alphabeta}.)

   This shows that $\varphi_t(1/z)/z$ is a Cauchy-Stieltjes transform of a unique compactly-supported probability measure $\nu_t$.
For a more detailed description of measure  $\nu_t$ and explicit formulas for its discrete and absolutely continuous components
we refer to \cite[Theorem 2.1]{Saitoh-Yoshida01}; see also Remark \ref{R1} below.

   It is well known that a  Cauchy-Stieltjes transform is an analytic function in the upper complex plane, determines measure uniquely, and if it extends to  real $z$ with $|z|$ large enough then the corresponding moment generating  function is well defined for all $|z|$ small enough and is given by $G_{\nu_t}(1/z)/z=\varphi_t(z)$.
This shows that $\varphi_t(z)$ is the moment generating function of the probability measure $\nu_t$.

Next we observe that \eqref{specform} in the operator notation is
$$\hh_t(x^n)=\frac{1+\eta x+\sigma x^2}{1+\sigma t}\sum_{k=1}^{n} c_{k}(t)x^{n-k} $$
 Writing   $c_k(t)=\int y^{k-1}\nu_t(dy)$, we therefore get
\begin{equation}\label{H-int}
\hh_t(f)(x)=\frac{1+\eta x+\sigma x^2}{1+\sigma t }\int \frac{f(y)-f(x)}{y-x}\nu_{t}(dy).
\end{equation}
 Since the operator version of relation \eqref{H2G} is $\gen_t(x^{n+1})=\hh_t(x^{n})+x\gen(x^n)$, we derive  \eqref{gen-free} from \eqref{H-int} by induction on $n$; for a similar reasoning see  \cite[Lemma 2.4]{Bryc-Wesolowski-2013-gener}.

\end{proof}

\begin{remark}
  \label{R1}
  Denote by  $\pi_{t,\eta,\theta,\sigma,\tau}(dx)$
  the univariate law of  $X_t$ for the free quadratic harness $(X_t)$ with parameters $\eta,\theta,\sigma,\tau$ as in \cite[Section 3]{Bryc-Matysiak-Wesolowski-05}. Then $\nu_t$ is given by
  \begin{equation}\label{nu_t}
\nu_{t}(dx)= \frac{1}{t(t+\tau)}(t^2+\theta t x + \tau x^2)\pi_{t,\eta,\theta,\sigma,\tau}(dx)\,.
\end{equation}

We read out this answer from  \cite[Eqtn. (3.4)]{Bryc-Matysiak-Wesolowski-05} using the following elementary relation between the Cauchy-Stieltjes transforms:

If $\nu(dx)=(a x^2+bx + c )\pi(dx)$ and $m=\int x\pi(dx)$ then the Cauchy-Stieltjes transforms of $\pi$ and $\nu$ are related
by the  formula
\begin{equation}
  \label{Jacka-ulubiony}
  G_\nu(z)=(az^2+bz+c) G_\pi(z)-am -  az- b.
\end{equation}
In our setting, $m=0$, $a=\frac{\tau}{t(t+\tau)}$, $b=\theta/(t+\tau)$, $c=t/(t+\tau)$, and \cite[Eqtn. (3.4)]{Bryc-Matysiak-Wesolowski-05} gives
\begin{multline}\label{duze-G}
G_{\pi}(z)=\frac{\tau z+\theta t}{\tau z^2+\theta t z+t^2}+\frac{t\left[(1+\sigma\tau+2\sigma t)z+t\eta-\theta\right]}{2(\sigma z^2+\eta z+1)(\tau z^2+\theta t z+t^2)}\\
-\frac{t\sqrt{\left[(1-\sigma\tau) z-(\alpha+\sigma\beta)t-\beta-\alpha\tau\right]^2-4(1+\sigma t)(t+\tau)(1+\alpha\beta)}}{2(\sigma z^2+\eta z+1)(\tau z^2+\theta t z+t^2)}.
\end{multline}
Inserting this expression into the right hand side of \eqref{Jacka-ulubiony} we get  the right hand side of \eqref{duze-G0}.
Uniqueness of Cauchy-Stieltjes transform implies  \eqref{nu_t}.

\end{remark}

\section{Some other cases}

Several other cases can be worked out by a similar technique based on \eqref{H}. Recasting \cite[Section IV]{Feinsilver:1978a} in our notation, for $m=0,1,\dots$ we have
\begin{equation}\label{r1r}
\ddd_1^{m+1}\fff-\fff \ddd_1^{m+1}=(m+1)\ddd_1^m,
\end{equation} and if
$$\hhh=\sum_{k=1}^{\infty}\,\tfrac{c_k}{k!}\,\ddd_1^k$$
then
\begin{equation}\label{AS}
\ggen=\sum_{m=1}^{\infty}\,\tfrac{c_m}{(m+1)!}\,\ddd_1^{m+1}.
\end{equation}

The generators for a L\'evy processes $(\xi_t)$  with exponential moments act on polynomials in variable $x$ via $\kappa(\frac{\partial}{\partial x})$, where   $t \kappa(\theta)=\log \E(\exp(\theta \xi_t))$; this well-known formula appears e.g. \cite[Section 3]{Anshelevich:2011}.
\subsection{Centered Poisson process}
For example, the quadratic harness with $\gamma=1$ and $\eta=\sigma=\tau=0$ which corresponds to the Poisson process can be also analyzed by the algebraic technique. Equation \eqref{GlownyWzorek}
  assumes the form
\begin{equation}\label{com2}
\hhh\fff-\fff\hhh=\eee+\theta \hhh
\end{equation}

\arxiv{
We search a solution  $\hhh$ of the form
$$
\hhh=\sum_{k=1}^{\infty}\,\tfrac{c_k}{k!}\ddd_1^k=\varphi(\ddd_1).
$$
Due to \eqref{r1r}  equation \eqref{com2} implies
$$
\varphi'(d)=1+\theta\varphi(d)
$$
and the solution with $\varphi(0)=0$ is $\varphi(d)=\tfrac{1}{\theta}\left(e^{\theta d}-1\right)$.

%$$
%\sum_{k=1}^{\infty}\,\tfrac{c_k}{(k-1)!}\,\ddd_1^{k-1}=\E+\theta\,\sum_{k=1}^{\infty}\,\tfrac{c_k}{k!}\ddd_1^k.
%$$
%Since the left hand side above can be written as
%$$
%c_1\eee+\sum_{k=1}^{\infty}\,\tfrac{c_{k+1}}{k!}\ddd_1^k
%$$
%comparing the coefficients of $\ddd_1^k$, $k=0,1,\ldots$, we get
%$$
%c_1=1\qquad\mbox{and}\qquad c_{k+1}=\theta c_k,\quad k=1,2,\ldots
%$$
%That is $c_k=\theta^{k-1}$, $k\ge 1$.

}
We get
$$
\hhh= \tfrac{1}{\theta}\left(e^{\theta\ddd_1}-\eee\right).
$$

It follows from \eqref{AS} that
\arxiv{
$$
\ggen=\sum_{m=1}^{\infty}\,\tfrac{\theta^{m-1}}{(m+1)!}\,\ddd_1^{m+1}=\tfrac{1}{\theta^2}\sum_{m=2}^{\infty}\,\tfrac{\theta^m}{m!}\,\ddd_1^m.
$$
Consequently,
}
$$
\ggen=\tfrac{1}{\theta^2}\left(e^{\theta\ddd_1}-\eee-\theta\ddd_1\right).
$$
\subsection{A non-L\'evy example}

Here we consider a quadratic harness with parameters $\gamma=1$, $\tau=\sigma=0$. This quadratic harness appeared under the name
{\em quantum Bessel process} in \cite{biane1996quelques} (see also \cite{matysiak2015zonal} for a multidimensional version), and as
{\em classical bi-Poisson} process in \cite{Bryc-Wesolowski-05}.

Then  equation \eqref{GlownyWzorek}  assumes the form
\begin{equation}\label{com4}
\hhh\fff-\fff\hhh=\eee+\theta\hhh+\eta(\fff-t\hhh)=\eee+\eta\fff+(\theta-t\eta)\hhh.
\end{equation}
For $\hhh=(\eee+\eta\fff)\tilde{\hhh}$ we obtain
$$
(\eee+\eta\fff)(\tilde{\hhh}\fff-\fff\tilde{\hhh})=(\eee+\eta\fff)(\eee+(\theta-t\eta)\tilde{\hhh}).
$$
Comparing this with \eqref{com2} we conclude that by uniqueness the solution of \eqref{com4} is
$$
\hhh=\tfrac{1}{\theta-t\eta}(\eee+\eta\fff)\left(e^{(\theta-t\eta)\ddd_1}-\eee\right).
$$

From \eqref{H2G-sol} it follows that
$$
\ggen=:\ggen(\hhh)=(\eee+\eta\fff)\,\ggen(\tilde{\hhh}).
$$
Therefore   we obtain
$$
\ggen=\tfrac{1}{(\theta-t\eta)^2}\,(\eee+\eta\fff)\left(e^{(\theta-t\eta)\ddd_1}-\eee-(\theta-t\eta)\ddd_1\right).
$$

\section*{Acknowledgement} %
 JW research was supported in part by NCN grant 2012/05/B/ST1/00554. WB research was supported in part by the Taft Research Center at the University of Cincinnati.

%%%%%%%%%%%%%%%%%%%% %BBL
 \def\cprime{$'$}

\end{document}